\documentclass[a4paper,11pt,reqno]{amsart}

\usepackage{latexsym} 
\usepackage{amssymb,amsmath,amsopn}
\usepackage{graphics}  
\usepackage{url}
\usepackage{array}
\usepackage{booktabs}

\newtheorem{theorem}{Theorem}[section]
\newtheorem{lemma}[theorem]{Lemma}

\newtheorem{corollary}[theorem]{Corollary}

\theoremstyle{definition}

\DeclareMathOperator{\Inf}{Inf}

\renewcommand{\theta}{\vartheta}

\newcommand{\Ind}{\big\uparrow}

\usepackage{color}

\begin{document}

\thispagestyle{empty}

\title{Decompositions of some twisted Foulkes characters} 
\author{Melanie de Boeck and Rowena Paget}
\date{30th August 2014}
\address{ School of Mathematics, Statistics \& Actuarial Science  \\
University of Kent \\
Canterbury \\
Kent CT2 7NF, UK \\
}
\email{mjc53@kent.ac.uk}  \email{r.e.paget@kent.ac.uk}

\begin{abstract}
We decompose the twisted Foulkes characters $\phi^{(2^n)}_\nu$, or equi\-valently the plethysm $s_\nu \circ s_{(2)}$, in the cases where $\nu$ has either two rows or two columns, or is a hook partition.
\end{abstract}

\maketitle

\thispagestyle{empty}

\section{Introduction}
The twisted Foulkes characters are a generalisation of the characters that are at the heart of the long-standing Foulkes' Conjecture~\cite{Foulkes}.
Let $m$ and $n$ be positive integers and, as usual, denote the irreducible characters of the symmetric group $S_n$ by $\chi^\nu$, for   $\nu$   a partition of $n$. Starting with $\chi^\nu$, we let $\Inf \chi^\nu$ denote the character of the wreath product $S_m \wr S_n$ obtained by inflating $\chi^\nu$  using the canonical surjection $S_m \wr S_n \to S_n$. Let
\[ \phi^{(m^n)}_\nu = \left( \Inf \chi^\nu \right) \Ind_{S_m \wr S_n}^{S_{mn}};
\]
we call these {\it twisted Foulkes characters}.   Translating into the language of symmetric functions, determining  the irreducible constituents of $\phi^{(m^n)}_\nu$ is equivalent to expressing the plethysm $s_\nu \circ s_{(m)}$ as a sum of Schur functions $s_\lambda$. Such decompositions are unknown except in a few special cases. When $\nu$ is a partition of 2, 3 or 4, or when $m=2$ and $\nu$ has a single row or column then the answer is given by the work of Littlewood~\cite{Littlewood}, Thrall~\cite{Thrall}, Foulkes~\cite{Foulkes2} and Howe~\cite{Howe}.

In this note we extend the known results when $m=2$ to give formulas for 
the decomposition of $\phi^{(2^n)}_\nu$ when $\nu$ has either two rows or two columns, or is a hook partition. The short proofs use elementary techniques from the representation theory of symmetric groups.

\section{Preliminaries}

We first recall the well-known decompositions of the $S_{2n}$-characters $\phi^{(2^n)}_{(n)}$ and $\phi^{(2^n)}_{(1^n)}$. For any partition $\alpha=(\alpha_1,\alpha_2, \ldots, \alpha_k)$ of $n$, define $2\alpha$ to be the partition of $2n$ obtained by doubling the length of each part of $\alpha$. Note that every partition of $2n$ which has all parts even can be written as $2\alpha$ for a unique $\alpha$. The following decomposition of $\phi^{(2^n)}_{(n)}$ goes back to Thrall~\cite{Thrall}:
\begin{equation}\label{eqn:F2}
 \phi^{(2^n)}_{(n)}= \sum \chi^{2\alpha}, 
\end{equation}
where the sum is over all partitions $\alpha$ of $n$.
If the partition $\alpha$ of $n$ has all its parts distinct then we may define the partition $2[\alpha]$ of $2n$ to be the partition whose leading diagonal hook lengths are $2\alpha_1, 2\alpha_2, \ldots, 2\alpha_k$ and whose $i$-th part is $\alpha_i+i$ for $i=1, \ldots, k$.  For example, $2[(5,2,1)]=(6,4^2,1^2)$.  The following decomposition of $\phi^{(2^n)}_{(1^n)}$ is also well-known; see for example~\cite[I.8, Ex. 6(d)]{Macdonald}:
\begin{equation} \label{eqn:Fsign2}\phi^{(2^n)}_{(1^n)}= \sum \chi^{2[\alpha]}, \end{equation}
where the sum is over all partitions $\alpha$ of $n$ with distinct parts.
\smallskip

In our work, we will repeatedly make use of the following lemma. The Littlewood--Richardson coefficients are denoted by  $c^\lambda_{\sigma, \tau}$ as usual for partitions $\sigma$, $\tau$ and $\lambda$. 

\begin{lemma} \label{lem:ind}
Let $\nu$ be a partition of $n-r$ and $\mu$ be a partition of $r$. Then
\[ \left( \phi_\nu^{(m^{n-r})} \times \phi_\mu^{(m^r)} \right) \Big\uparrow_{S_{m(n-r)}\times S_{mr}}^{S_{mn}} = \sum c^\lambda_{\nu, \mu} \phi^{(m^n)}_\lambda, \]
where the sum is over all partitions $\lambda$ of $mn$.
\end{lemma}

\begin{proof}
The proof is straightforward by the properties of induction and~\cite[Lemma~3.3(1)]{ChuangTan}. Alternatively, phrased using symmetric functions, it follows since $f \mapsto f \circ s_{(m)}$ is a ring homomorphism; see for example \cite[Equation~(8.3), I.8]{Macdonald}.
\end{proof}

\section{Explicit Decompositions}
We now record the irreducible constituents of some twisted Foulkes characters.

\begin{theorem}\label{thm:2row}
The character $\phi^{(2^n)}_{(n-r,r)}$ decomposes as
\[\phi^{(2^n)}_{(n-r,r)} = \sum_{\lambda  } \sum_{\alpha, \beta, \gamma, \delta} \left( c^\lambda_{2\alpha, 2\beta} -c^\lambda_{2\gamma,2\delta} \right)\chi^\lambda, \]
where the second sum is over all partitions $\alpha$ of $n-r$, $\beta$ of $r$, $\gamma$ of $n-r+1$ and $\delta$ of $r-1$.
\end{theorem}

\begin{proof}
We apply Lemma~\ref{lem:ind} twice  to obtain
\begin{eqnarray*}
\left( \phi_{(n-r)}^{(2^{n-r})} \times \phi_{(r)}^{(2^r)}\right)\Big\uparrow^{S_{2n}}
&=& \phi_{(n-r,r)}^{(2^n)}+ \phi_{(n-r+1,r-1)}^{(2^n)}+\cdots + \phi_{(n-1,1)}^{(2^n)}+ \phi_{(n)}^{(2^n)}\\
&=& \phi_{(n-r,r)}^{(2^{n})} + \left( \phi_{(n-r+1)}^{(2^{n-r+1})} \times \phi_{(r-1)}^{(2^{r-1})}\right) \Big\uparrow^{S_{2n}}.
\end{eqnarray*}
We now apply Equation~(\ref{eqn:F2}) to conclude that
\[\phi_{(n-r,r)}^{(2^{n})} = \Bigg( \sum_{\alpha \vdash n-r} \chi^{2\alpha} \times \sum_{\beta \vdash r} \chi^{2\beta}\Bigg)\Bigg\uparrow^{S_{2n}} - \Bigg( \sum_{\gamma \vdash n-r+1} \chi^{2\gamma} \times \sum_{\delta \vdash r-1} \chi^{2\delta}\Bigg)\Bigg\uparrow^{S_{2n}}. \]
An application of the Littlewood--Richardson rule yields the statement.
\end{proof}

The description of the character $\phi^{(2^n)}_{(n-1,1)}$ is particularly simple. We shall denote the number of distinct parts of a partition $\lambda$ by $a_\lambda$.
\begin{corollary}\label{cor:1}
The decomposition of $\phi^{(2^n)}_{(n-1,1)}$ into its irreducible con\-sti\-tu\-ents is
\[ \phi^{(2^n)}_{(n-1,1)} = \sum_{\gamma \vdash n} (a_{2\gamma}-1) \chi^{2\gamma} + \sum_\mu \chi^\mu, \]
where the second sum is over all partitions $\mu$ of $2n$ 
which have all but two parts even and the two odd parts distinct.
\end{corollary}

\begin{proof}
The Littlewood--Richardson coefficient $c^\lambda_{2\alpha,(2)}$ is one precisely if we obtain $\lambda$ from $2\alpha$ by adding two boxes that do not lie in the same column, and otherwise it is zero. Constituents are therefore labelled either by even partitions or by partitions with precisely two odd parts that are distinct. The latter appear with multiplicity one. In the former case, the multiplicity is as stated above since the positions of the added boxes must be at the end of the first row of its length in $2\alpha$.
\end{proof}

The description of $\phi^{(2^n)}_\nu$ in the case where $\nu$ has two columns is obtained similarly.

\begin{theorem}\label{thm:2col}
The character $\phi^{(2^n)}_{(2^r,1^{n-2r})}$ decomposes as
\[\phi^{(2^n)}_{(2^r,1^{n-2r})}=\sum_{\lambda}\sum_{\alpha,\beta,\gamma,\delta}\left( c^\lambda_{2[\alpha],2[\beta]}-c^\lambda_{2[\gamma],2[\delta]}\right)\chi^\lambda,\]
where the second sum is over all partitions with no repeated parts $\alpha$ of $n-r$, $\beta$ of $r$, $\gamma$ of $n-r+1$ and $\delta$ of $r-1$.
\end{theorem}

\begin{proof}
By Lemma~\ref{lem:ind} we see that 
\[\left(\phi^{(2^{n-r})}_{(1^{n-r})}\times\phi_{(1^r)}^{(2^r)}\right)\Big\uparrow_{S_{2(n-r)}\times S_{2r}}^{S_{2n}}=\sum_{j=0}^r\phi_{(2^j,1^{n-2j})}^{(2^{n})}\]
and therefore
\begin{multline} \phi^{(2^{n})}_{(2^r,1^{n-2r})} = \left( \phi^{(2^{n-r})}_{(1^{n-r})} \times \phi_{(1^r)}^{(2^r)}\right)\Big\uparrow_{S_{2(n-r)}\times S_{2r}}^{S_{2n}}\\ - \left( \phi^{(2^{n-r+1})}_{(1^{n-r+1})} \times \phi_{(1^{r-1})}^{(2^{r-1})}\right) \Big\uparrow_{S_{2(n-r+1)}\times S_{2(r-1)}}^{S_{2n}}.\end{multline}
The result is obtained using Equation~(\ref{eqn:Fsign2}) and the Littlewood--Richardson rule.
\end{proof}

The simple decomposition of $\phi^{(2^n)}_{(2,1^{n-2})}$ that we obtain is recorded below. We require one additional piece of notation: for a partition $\gamma$ let\linebreak $b_\gamma =\vert\{i:\gamma_i>\gamma_{i+1}+1\}\vert$.

\begin{corollary} \label{cor:2}
The decomposition of $\phi^{(2^n)}_{(2,1^{n-2})}$  into its irreducible con\-sti\-tu\-ents is
\[ \phi^{(2^n)}_{(2,1^{n-2})} = \sum_{\gamma } (b_\gamma-1) \chi^{2[\gamma]} + \sum_\mu \chi^\mu,
 \]
where the first sum is over all partitions $\gamma$ of $n$ with distinct parts, and the second sum is over all partitions $\mu$ of $2n$ which are obtained from adding two nodes to a partition of the form $2[\alpha]$ such that the two nodes do not lie in the same column and do not lie at opposite ends of any leading diagonal hook.
\end{corollary}

Finally, we turn our attention to $\phi^{(2^n)}_\nu$ in the case where $\nu$ is a hook partition.

\begin{theorem}\label{thm:hook}
The following two formulas decompose   $\phi^{(2^n)}_{(n-r,1^r)}$  into its irreducible constituents:
\begin{align*} \phi^{(2^n)}_{(n-r,1^r)} &= \sum_\lambda  \left(  \sum_{j=0}^r (-1)^j \sum_{\alpha^{(j)}, \beta^{(j)}} c^\lambda_{2\alpha^{(j)},\,2[\beta^{(j)}]}  \right) \chi^\lambda,
\intertext{where the third summation is over all partitions $\alpha^{(j)}$ of $n-r+j$ and over all partitions $\beta^{(j)}$ of $r-j$ with distinct parts;}
 \phi^{(2^n)}_{(n-r,1^r)} &= \sum_\lambda \left(\sum_{j=1}^{n-r}(-1)^{j-1} \sum_{\gamma^{(j)},\,\delta^{(j)}} c^\lambda_{2\gamma^{(j)},\,2[\delta^{(j)}]} \right) \chi^\lambda,
\end{align*}
where the third summation is over all partitions $\gamma^{(j)}$ of $n-r-j$ and over all partitions $\delta^{(j)}$ of $r+j$ with distinct parts.
\end{theorem}

\begin{proof}
 Lemma~\ref{lem:ind} tells us that 
\[\left(\phi^{(2^{n-r})}_{(n-r)}\times\phi_{(1^r)}^{(2^r)}\right)\Big\uparrow_{S_{2(n-r)}\times S_{2r}}^{S_{2n}}=\phi_{(n-r,1^{r})}^{(2^{n})}+\phi_{(n-r+1,1^{{r-1}})}^{(2^{n})},\]
and hence
\[\phi^{(2^n)}_{(n-r,1^r)}=\sum_\lambda\left(\sum_{\alpha^{(0)},\,\beta^{(0)}}c^\lambda_{2\alpha^{(0)},\,2[\beta^{(0)}]} \right)\chi^\lambda-\phi^{(2^n)}_{(n-r+1,1^{r-1})}.\]
Repeatedly using this relation, along with Equations~(\ref{eqn:F2}) and~(\ref{eqn:Fsign2}) and the Littlewood--Richardson rule, gives the first statement. The second statement follows in exactly the same way starting instead from
\[ \phi^{(2^n)}_{(n-r,1^r)}=\left(\phi^{(2^{n-r-1})}_{(n-r-1)}\times\phi_{(1^{r+1})}^{(2^{r+1})}\right)\Big\uparrow_{S_{2(n-r-1)}\times S_{2(r+1)}}^{S_{2n}}-\phi^{(2^n)}_{(n-r-1,1^{r+1})}.\qedhere \] 
\end{proof}

Other statements similar to Corollaries~\ref{cor:1} and~\ref{cor:2} can be obtained by some careful analysis of the Littlewood--Richardson coefficients. For example, Table~\ref{table1} below shows the constituents of  $\phi^{(2^n)}_{(n-2,1^2)}$ and $\phi^{(2^n)}_{(n-2,2)}$; here we set $\{\lambda\}=\{\lambda_i:\lambda_i>0\}$ and define $r_\lambda:=\vert\{k: k\text{ is a repeated part of }\lambda\}\vert$, and for $X,Y\subseteq\mathbb{N}$, $$N_\lambda(X|Y):=\{k\geq0: 2k+x\in\{\lambda\}\,\forall\, x\in X,\,2k+y\notin\{\lambda\}\,\forall\, y\in Y\}.$$
Recall that $a_\lambda$ denotes the number of distinct parts of a partition $\lambda$.
\medskip

\begin{table}\caption{The multiplicities of the constituents of $\phi^{(2^n)}_{(n-2,1^2)}$ and $\phi^{(2^n)}_{(n-2,2)}$.}\label{table1}

\begin{tabular}{@{}m{3.5cm}cc@{}} \toprule
\centering $\lambda$ & \multicolumn{2}{c}{Multiplicity of $\chi^\lambda$ as a constituent of $\phi^{(2^n)}_\nu$}  \\ \cmidrule(r){2-3}
&$\nu=(n-2,1^2)$&$\nu=(n-2,2)$\\
\midrule
$\lambda$ has all parts even & $\binom{a_\lambda}{2}-a_\lambda+1$ & $a_\lambda(a_\lambda-2)+N_\lambda(4|2)+r_\lambda$\\ \addlinespace
$\lambda$ has 2 odd parts that are distinct, and all other parts even & \parbox{3.5cm}{\centering $N_\lambda(3|2)+2N_\lambda(2|1)$ $+N_\lambda(1,2|\emptyset)-1$} & \parbox{3.5cm}{\centering $2N_\lambda(2|1) + N_\lambda(1,2|\emptyset)$  $+ N_\lambda(3|1,2)-1$}\\ \addlinespace 
$\lambda$ has 2 equal odd parts and all other parts even & $N_\lambda(3|2)+N_\lambda(2|1)$ & 0 \\ \addlinespace
$\lambda$ has 4 odd parts that are distinct and all other parts even & 3 & 3\\ \addlinespace
$\lambda$ has 4 odd parts, one repeated and two distinct, and all other parts even & 1 & 1 \\ \addlinespace
$\lambda$ has 4 odd parts, forming two pairs of equal odd parts, and all other parts even & 0 & 1 \\ \addlinespace
$\lambda$ not of the above form & 0 & 0\\
\bottomrule
\end{tabular}
\end{table}

Our results decompose the plethysms  $s_\nu \circ s_{(2)}$ when $\nu$ has two rows or two columns or is a hook partition. We remark that  by applying the $\omega$ involution (see \cite[Ch.~I, Equation~(2.7)]{Macdonald})  the plethysms $s_\nu \circ s_{(1^2)}$ (and the corresponding characters of $S_{2n}$) are also determined since 
\cite[Ch.~I, Equation~(3.8) and \S 8, Example 1(a)]{Macdonald} tell us that $\omega(s_\nu \circ s_{(2)}) = s_{\nu} \circ s_{(1^2)}$ and $\omega(s_\lambda)=s_{\lambda'}$.

\section*{Acknowledgments}
The first author gratefully acknowledges the financial support provided by the School of Mathematics, Statistics and Actuarial Science at University of Kent, and the Engineering and Physical Sciences Research Council (grant
number EP/P505577/1).

\end{document}